\documentclass[12pt]{amsart}
\usepackage{graphicx}

\newtheorem{thm}{Theorem}[section]

\newtheorem{defn}[thm]{Definition}
\newtheorem{lemma}[thm]{Lemma}
\newtheorem{conj}[thm]{Conjecture}

\newtheorem{remark}[thm]{Remark}

\usepackage{amsmath}
\usepackage{amsxtra}
\usepackage{amscd}
\usepackage{amsthm}
\usepackage{amsfonts}
\usepackage{amssymb}
\usepackage{eucal}

\newcommand{\g}{{\mathfrak{g}}}

\renewcommand{\sl}{{\mathfrak{sl}}}

\newcommand{\ch}{{\rm ch}}
\newcommand{\C}{{\mathbb C}}
\newcommand{\Z}{{\mathbb Z}}
\newcommand{\N}{{\mathbb N}}
\newcommand{\Hom}{{\rm Hom}}

\newcommand{\bm}{{\mathbf m}}
\newcommand{\bn}{{\mathbf n}}
\newcommand{\bx}{{\mathbf x}}
\newcommand{\by}{{\mathbf y}}

\newcommand{\al}{{\alpha}}
\newcommand{\tQ}{{\widetilde{Q}}}
\newcommand{\sgn}{{\rm sgn}}

\numberwithin{equation}{section}

\begin{document}
\title{$Q$-systems as cluster algebras}
\author{Rinat Kedem}
\address{Department of Mathematics, University of Illinois, 1409 West
  Green Street, Urbana, IL 61821. e-mail: rinat@uiuc.edu}
\date{\today}
\begin{abstract}
$Q$-systems first appeared in the analysis of the Bethe equations for
  the $XXX$-model and generalized Heisenberg spin chains
  \cite{KRqsys}. Such systems are known to exist for any simple Lie
  algebra and many other Kac-Moody algebras.  We formulate the
  $Q$-system associated with any simple, simply-laced Lie algebras
  $\g$ in the language of cluster algebras \cite{FZcluster}, and
  discuss the relation of the polynomiality property of the solutions
  of the $Q$-system in the initial variables, which follows from the
  representation-theoretical interpretation, to the Laurent
  phenomenon in cluster algebras \cite{FZlaurent}.
\end{abstract}

\maketitle
\section{Introduction}
For any simple Lie algebra, and for many other Kac-Moody algebras,
there exists an associated system of equations known as a $Q$-system. These equations can be regarded as 
a recursion relation in $i$ for the variables
$\{ Q_{\al,i}\ : \ \al\in I_r, i\in \Z_+\}$, $I_r=\{1,...,r\}$ where $r$ is
the rank of the algebra $\g$. 

Solutions to the $Q$-system describe characters of special representations of the quantum affine algebra $U_q(\widehat{\g})$ or of the Yangian $Y(\g)$, as well as the classical limit described by Chari \cite{Chari}.

The $Q$-systems can be considered as a type of discrete dynamical
system, and for the root system of type $A$, it is  a specialization of the discrete Hirota equations. The $Q$-system first
appeared in the analysis of Bethe Bethe ansatz of the generalized
inhomogeneous Heisenberg spin chain \cite{KRqsys} in the thermodynamic limit. This system is instrumental to the solution of the  Kirillov-Reshetikhin (KR) conjectures about
completeness of the Bethe ansatz solutions and about the structure of
finite-dimensional representations of the
quantum algebra \cite{Nakajima,Hernandez}.

The $Q$-systems have the remarkable property that their solutions are
polynomials, given appropriate initial data.  This
fact was the essential ingredient in the recent proof \cite{DFK} of
the combinatorial Kirillov-Reshetikhin conjectures, which are the completeness
conjectures for the generalized Heisenberg spin chains.

On the other hand, the concept of cluster algebras \cite{FZcluster} is
particularly well-suited to describing precisely this sort of system,
and it is clear from the definition of cluster algebras that the two
must be connected (this was also remarked in \cite{Her07}). In fact,
$Y$-systems, which are closely connected to $Q$-systems \cite{KNS},
have previously been studied using the cluster algebra formalism
\cite{FZysys}.

In this note, we explain how to formulate this relation
precisely in the case of simply-laced Lie algebras. The result is the
description of the $Q$-systems as a subgraph of a cluster algebra
tree which has particularly nice properties.

Section 2 is a short exposition of the background of the $Q$-system. In Section 3, we recall some basic facts about cluster algebras, and show that the $Q$-system defines a particular case of a cluster algebra. We also discuss the
relation of the polynomiality property, which can be proven
using purely representation-theoretical arguments
\cite{Nakajima,Hernandez}, to the remarkable property of cluster
algebras known as the Laurent phenomenon \cite{FZlaurent}.

\section{Background}
\subsection{The generalized Heisenberg spin chains and the
  Kirillov-Reshetikhin conjectures} 
  
 Below we make reference to the Bethe equations, but since we will use only certain combinatorial constructions derived from these equations rather than the equations themselves, we refer the reader to the literature for further information on them \cite{R,KR}. 
  
  The generalized inhomogeneous
Heisenberg model associated with the Yangian $Y(\g)$ of a simple Lie
algebra $\g$ was formulated in \cite{KRqsys} for the classical Lie
algebras. It can similarly be formulated for the quantum affine algebra
$U_q(\widehat{\g})$. This latter generalization corresponds to the
deformation of the XXX model to the XXZ model.

Let $R$ be either the rational $R$-matrix corresponding to the Yangian, or the  trigonometric $R$-matrix,
corresponding to the  quantum affine algebra. The transfer
matrix of the model is the trace over the auxilliary space $V_0$ of a
product of $R$-matrices:
\begin{equation}\label{transfer}
T_{V_0}(z_1,...,z_N) = {\rm Trace}_{V_0} R_{V_1,V_0}(z_1)\cdots
R_{V_N,V_0}(z_N).
\end{equation}

Here, the $R$-matrix $R_{V,V'}$ is the intertwiner of certain special
finite-dimensional representations of the Yangian $Y(\g)$ or quantum
affine algebra, which is obtained via fusion from the fundamental
representations. 

The inhomogeneity in the model is both in the spectral parameters
$z_i$ associated with each lattice site $i$, as well as in the
representations $V_i(z_i)$ at each site.  

However each of the representations $V_i(z_i)$ are assumed to be of
Kirillov-Reshetikhin (KR) type \cite{KRqsys}. This type of representation
can be defined, for example, by means of its Drinfeld polynomials,
although the modules were originally defined in terms of fusion of
$R$-matrices instead. 

One way to define (and describe) Kirillov-Reshetikhin modules for $Y(\g)$
for any Lie algebra $\g$ is to say that it is the
smallest irreducible $Y(\g)$-module with $\g$-highest weight which is
proportional to one of the fundamental weights. Therefore, KR-modules
are labeled by a triple, $\al\in I_r$ where $I_r$ are the labels of
the nodes in the Dynkin diagram, $i\in \Z_+$ a non-negative integer,
and $z\in \C^*$ a complex number, corresponding to the localization
parameter. The $\g$-highest weight of such a module is $i\omega_\al$
where $\omega_\al$ is the fundamental weight corresponding to the
label $\al$, and the module can be denoted by $V_{\al,i}(z)$.

Kirillov and Reshetikhin were able to describe the structure of the spaces $V_{\al,i}(z)$ as $\g$-modules using the Bethe ansatz equations for such models. The results turned out to have deep combinatorial and
representation-theoretical implications. 

The Hilbert space of the model
\eqref{transfer} is
\begin{equation}\label{Hilbert}
\mathcal H = \underset{i=1}{\overset{N}{\otimes}} V_i,
\end{equation}
where $N$ is the number of lattice sites.  The {\em completeness
hypothesis} is that the Bethe vectors provide a basis for the set of
$\g$-highest weight vectors in the Hilbert space, in the following sense.

Note that we always have $\g\subset Y(\g)$, or
$U_q(\g)\subset U_q(\widehat{\g})$ in the case of the trigonometric
$R$-matrix. Thus, the Hilbert space $\mathcal H$ decomposes into a
direct sum of finite-dimensional $\g$
or $U_q(\g)$-modules:
$$
\mathcal H \underset{\g-{\rm mod}}{\simeq} \underset{\lambda\in
  P^+}{\oplus} V(\lambda)^{\oplus m_{\lambda,\mathcal H}},
$$ where $V(\lambda)$ is the irreducible $\g$-module (or
$U_q(\g)$-module) characterized by the highest weight $\lambda$. Here,
$P^+$ is the set 
of dominant integral weights of $\g$. 

The multiplicity is the
dimension of the space of $\g$-linear homorphisms,
\begin{equation}\label{dimhom}
m_{\lambda,\mathcal H} = \dim\left( \Hom_\g\left(\mathcal H,\
V(\lambda)\right)\right). 
\end{equation}

The completeness conjecture is that the number of linearly independent
Bethe vectors characterized by $\lambda$ is equal to
$m_{\lambda,\mathcal H}$.

Inasmuch as the dimensions of the representations, and hence the
counting arguments, are concerned, there is no dependence on whether
the $R$-matrix is rational or trigonometric (assuming that $q$ is
generic, in particular, that it is not a root of unity). 
There is no dependence on the choice of auxilliary space, since it
does not influence the dimension of the Hilbert space.

For each $\lambda$, assuming the string hypothesis for the solutions
to the Bethe equations\footnote{Although the string
hypothesis is known to fail in various scenarios, the counting
algorithm of \cite{KR} turns out to give the correct number of states
nonetheless.}, the Bethe integers parametrize the Bethe vectors. 
These are sets of distinct integers, one set for each pair $(\al, i)$
where $\al\in I_r$ corresponds to one of the simple roots of the
underlying algebra $\g$, and $i\in \N$ is the length of the
``string''. These integers correspond to the string centers. 

Given a set $\{n_{\al,i}\in \Z_+~:~\al\in I_r, i\in \N\}$,
choose a set of $m_{\al,i}$ distinct integers are chosen from the interval
$[0,p_{\al,i}]$, where
\begin{equation}\label{vacancy}
p_{\al,i} = \sum_{j} \min(i,j)n_{\al,j}  - \sum_{\beta}{\rm
sgn}(C_{\al,\beta}) \sum_j \min(|C_{\al,\beta}| j, |C_{\beta,\al}|i)
m_{\beta,j}, \quad(\al\in I_r, i\in \N).
\end{equation}
These so-called {\em vacancy numbers} are determined from the large
$N$ analysis (where $N$ is the number of sites) of the thermodynamic
Bethe ansatz equations. The integers $n_{\al,i}$ are parameters of the
model -- they parametrize the representations $V_i(z_i)$ in the
definition of the model and hence $\mathcal H$.

There are two restrictions on the choice of numbers $\{m_{\al,i}\}$. The
first is clearly that the vacancy numbers should be non-negative,
\begin{equation}\label{ppositive}
p_{\al,i}\geq 0.
\end{equation} 
The second keeps track of the sector $\lambda$ to
which the solutions belong:
\begin{equation}\label{spin}
\sum_{i} i n_{\al,i} - \sum_{\beta,i} i C_{\al,\beta} m_{\beta,i} = l_\al,
\end{equation}
where the weight is $\lambda = \sum_{\al} l_\al \omega_\al$
and $\omega_\al$ are the fundamental weights of $\g$. Here, $C$ is the
Cartan matrix of the Lie algebra $\g$, with the convention that
$$
C_{ij} = \frac{2 (\al_i,\al_j)}{(\al_i,\al_i)},\quad \al_k\in \Pi.
$$

The number of ways to choose $m$ distinct integers from the interval
$[0,p]$ is the binomial coefficient:
\begin{equation}\label{binomial}
{m+p\choose m} = \frac{(p+m)(p+m-1)\cdots (p+1)}{m!}.
\end{equation}
Thus, for fixed $\lambda$ and $\bn=(n_{\al,i})_{i\in \N; \al\in I_r}$,
the number of solutions to the Bethe equations is
\begin{equation}\label{M}
M_{\lambda;\bn} = \sum_{\underset{p_{\al,i}\geq
    0}{m_{\al,i}\geq 0}} \prod_\al  
{m_{\al,i}+p_{\al,i}\choose m_{\al,i}},
\end{equation}
where the sum is taken over all non-negative integers
$\bm=(m_{\al,i})$ such that equation \eqref{spin} holds.

The completeness conjecture is therefore the following:
\begin{conj}\cite{KR}\label{KRconj}
The dimension of the space of $\g$-linear homomorphisms \eqref{dimhom} is
$$
m_{\lambda;\mathcal H} = M_{\lambda,\bn}
$$ where $\bn=(n_{\al,i})$ and 
$\bn_{\al,i}$ is the number of Kirillov-Reshetikhin type
modules $V_j$ in the tensor product \eqref{Hilbert} which have a
$\g$-highest $i \omega_\al$.
\end{conj}

Until recently, a proof of this conjecture was available in various
special cases, by proving a bijection between the set of ``rigged
configurations'' which enumerate the Bethe integers of the form
described above, and crystal paths (see \cite{KR,KS,KSS} for
example).

However, a completely general
proof is available for the following statement:
\begin{thm}\label{KRtwo}\cite{Nakajima,Hernandez}
The dimension $m_{\lambda;\mathcal H}$  of the space of homomorphisms
is equal to the so-called $N$-sum \cite{KR,HKOTY}:
\begin{equation}\label{N}
N_{\lambda;\bn} = \sum_{m_{\al,i}\geq 0} \prod_\al  
{m_{\al,i}+p_{\al,i}\choose m_{\al,i}},
\end{equation}
where the summation is taken over all non-negative integers $\bm$
subject to the restriction \eqref{spin}, but not subject to the
positivity condition \eqref{ppositive}.
\end{thm}
Note that the binomial coefficient with the definition
\eqref{binomial} is perfectly well-defined for values of $p<0$ as long
as $m<-p$. In fact, there is an identity,
$$
{m+p \choose m} = (-1)^m {-p-1\choose m}.
$$

Although at first glance, there is little difference between the
$M$-sum \eqref{M} and the $N$-sum \eqref{N} it is, in fact, a rather
subtle combinatorial identity. The obvious difference is that the
$M$-sum is a summation over manifestly non-negative terms, where in
the $N$-sum, there are many negative terms, and many cancellations
take place. However, the two sums are equal:
\begin{thm}\label{MN}\cite{DFK}. 
\begin{equation}
M_{\lambda;\bn}=N_{\lambda;\bn}.
\end{equation}
\end{thm}
This provides a proof of the combinatorial
Kirillov-Reshetikhin conjecture \ref{KRconj} for all simple Lie algebras $\g$.

\subsection{$Q$-systems}

Theorems \ref{KRtwo} and \ref{MN} are, in fact, corollaries of a more
basic fact, as was shown by \cite{HKOTY}, about the nature of the
solutions of an associated $Q$-system.

First, let us describe the notion of polynomiality in the case that
$\g=\sl_2$. In this case, Kirillov-Reshetikhin modules are in
one-to-one correspondence with the irreducible finite-dimensional
representations of $\sl_2$ and have the same dimension. That is, they
are irreducible as $\sl_2$-modules if we start with the Yangian
module, or as a $U_q(\sl_2)$-module if we refer to the quantum
affine algebra.

Let $Q_0=1$ be the character of the one-dimensional representation on
which $\sl_2$ acts trivially, and $Q_1$ the character of the defining
or fundamental representation of dimension 2. For simplicity, define
$t=Q_1$. The character of the $j+1$-dimensional irreducible
representation, with highest weight $j \omega_1$, is then known to be
polynomial -- in fact,
a Chebyshev polynomial of the second kind -- in $t$. This is just a
statement of the fact that any irreducible representation of $\sl_2$
is a polynomial representation, that is, it can be obtained from the
image of the Young symmetrizer acting on the tensor product of the
fundamental representation.

One can show that the Chebyshev polynomials $Q_j$, or, equivalently,
the characters of the $j+1$-dimensional irreducible representations of
$\sl_2$, satisfy a recursion relation
$$
Q_{j+1}Q_{j-1} + 1 = Q_{j}^2.
$$
This is the $Q$-system for $\sl_2$. 

It turns out that for any simple Lie algebra there is a similar system
of recursion equations. For simply-laced simple Lie algebras, we write
$\beta\sim \al$ if the nodes $\beta$ and $\al$ are  connected in the
Dynkin diagram of $\g$. Then 
the $Q$-system for a simply-laced Lie algebra $\g$ is
\begin{equation}\label{Qsystem}
Q_{\al,k+1}= Q_{\al,k-1}^{-1}\left(Q_{\al,k}^2-\prod_{\beta\sim\al}
Q_{\beta,k}\right),\ (k>0), \qquad Q_{\al,0}=1,\quad Q_{\al,1}=t_\al.
\end{equation}
There are also $Q$-systems for any other simple Lie algebra
\cite{KR,HKOTY} and for twisted affine algebras \cite{HKOTT}. We do
not discuss these systems in this paper. 

We set the formal parameter $t_\al$ to be the $\g$-character of the
``fundamental KR-module'' $Q_{\al,1}$,
corresponding to highest weight $\omega_\al$. In general, this is not
necessarily an irreducible $\g$-module. For example, if $\g=D_4$,
$$
Q_{2,1} = \ch V({\omega_2}) + \ch V(0),
$$
where $V(0)$ is the trivial representation. These modules are also
known as fundamental representations of Yangians (or minimal
affinizations) and were studied in \cite{Kleber,ChariKR}.

Under fairly mild asymptotics conditions \cite{HKOTY}, it can be shown
that the solutions of the $Q$-systems are characters of the
KR-modules. This has been done for the special case
above by Nakajima \cite{Nakajima} and in complete generality by
Hernandez \cite{Her07}.

In the case of $\sl_n$, KR-modules are in fact the irreducible $\sl_n$
representations with ``rectangular highest weights''
the form $i \omega_\al$. 

In the case where $\g$ is not of type $A$, the characters in question
are sometimes not irreducible as $\g$-modules. However, they have a
triangular decomposition in terms of irreducible characters with smaller
highest weights, in the sense of partial ordering of highest weights.
That is,
\begin{equation}\label{decomposition}
V_{\al,i}(z)~\underset{\g}{\simeq}~ V(i\omega_\al) \oplus\ \left(
\underset{\lambda< i \omega_\al}{\oplus}
 V(\lambda)^{m_{\lambda;i\omega_\al}}\right).
\end{equation}
The decomposition multiplicities are given by the
Kirillov-Reshetikhin conjecture.

In \cite{HKOTY} it was shown that Theorem \ref{KRtwo} is a consequence
of the fact that the characters of KR-modules solve the $Q$-system. 

In proving \ref{MN}, an essential ingredient was this fact, and the
following Lemma, which follows entirely from
representation-theoretical arguments and the theorem of
\cite{Nakajima} or more generally \cite{Hernandez} (see also \cite{KleberPoly}): 
\begin{lemma}\label{polynomiality}\cite{DFK}
The solutions
of the $Q$-system, for any Lie algebra, are a polynomial in the
variables $\{ Q_{\al,1}~|~\al\in I_r\}$.
\end{lemma}
\begin{proof} This follows from
\eqref{decomposition}, which shows that
the transition matrix between KR-modules and irreducible $\g$-modules
is unitriangular and hence invertible. Thus, any KR-module is in the
Groethendieck ring generated by the fundamental and trivial representations of
$\g$. That is, the character $Q_{\al,i}$ is a polynomial in the
characters of the fundamental $\g$-modules. Finally, using
the relation \eqref{decomposition} again, it is seen to be a
polynomial in $Q_{\beta,1}$.
\end{proof}

This statement can be rephrased as follows: The solutions of the
$Q$-system are polynomials in the variables $\{Q_{\al,1}\}_{\al\in
I_r}$ in the limit $Q_{\al,0}\to 1$.  In general, we do not know how
to show this polynomiality property directly from the $Q$-system.

The proof of this statement relies entirely on the fact that the
characters of KR-modules are solutions of the $Q$-system. However, the
$Q$-system can be viewed as a dynamical system, or a recursion
relation, independently of this fact. We show below that it can be
expressed as a cluster algebra. 

Seen in this light, the polynomiality property appears to be closely
connected to the Laurent phenomenon \cite{FZlaurent} for cluster
algebras (this was also recently suggested by Hernandez,
\cite{Her07}). The rest of this note is devoted to describing the
cluster algebra structure of $Q$-systems, as a first step towards
understanding this phenomenon. We will show in Section \ref{Qcluster} 
that the $Q$-system for
simply-laced Lie algebras is, in fact, a quotient of a subgraph of a
cluster algebra.

\section{$Q$-systems as Cluster algebras}
\subsection{Cluster algebras}
For the definition of a cluster algebra, we refer the reader to the
excellent summary in \cite{FZfour}. We recap this definition briefly
here, specialized to the particularly simple case we consider in the
context of $Q$-systems corresponding to simply-laced simple Lie
algebras.

First, consider a regular $n$-ary tree $\mathbb T_n$. This is a tree
with certain nodes, each node being connected to $n$ other nodes via
undirected, labeled edges. Each node is connected to edges labeled by
the distinct labels $1,...,n$.

At each node there is a seed of the form $(\bx,B)$. (More
generally, there are also coefficients $\by$ at each node, but in the case
under consideration here, it is possible to take a normalized cluster
algebra, with all coefficients are $1$, so we do not consider the more
general case of nontrivial coefficients.)
Here, $\bx$ is an
$n$-vector with entries $x_i$ and $B$ is a skew-symmetric $n\times n$ integer
matrix with entries $B_{ij}$. When
we want to be explicit we will label the seed by the node to which it
corresponds, e.g. $(\bx,B)_t$ is the seed at node $t$.

Nodes connected by an edge labeled $k$ correspond to a mutation of
the seed, or an evolution of the system in the direction $k$. This
mutation is denoted by a map $\mu_k: (\bx,B)\mapsto (\bx',B')$. The
map acts as follows on the seed:

\begin{eqnarray}
\mu_k(x_i) &=& \left\{ \begin{array}{ll} x_i, & i\neq k; \\
x_k^{-1}(\prod_{j=1}^n x_j^{[B_{jk}]_+} + \prod_{j=1}^n
x_j^{[-B_{jk}]_+}),& i=k.\end{array}\right. \label{mutation}\\
\mu_k(B_{ij}) &=& \left\{ \begin{array}{ll}
-B_{ij} & \hbox{if $i=k$ or $j=k$};\\
B_{ij}+{\rm sign}(B_{ik})[B_{ik}B_{kj}]_+, &
\hbox{otherwise}.\end{array}\right. 
\end{eqnarray}
Here, $[n]_+$ means the positive part of the integer $n$. It is equal
to $0$ if $n\leq 0$, and is equal to $n$ if $n>0$.

One can immediately check that $\mu_i\circ \mu_i(\bx,B)=(\bx,B)$,
which is why the edges of the tree are not oriented.

A cluster pattern is an assignment of a labeled seed to each node $t$
in $\mathbb T_n$, such that the seeds at connected nodes are related by the
relevant seed mutations.  A cluster algebra is the algebra in the
cluster variables $\bx_t$ which are related by such mutations. 

Thus, one has an $n$-ary tree $\mathbb T_n$ and an evolution on it,
connecting its nodes along the edges determined by the mutation matrix $B$.
Starting from any node, one can compute the cluster variables $\bx$
and the mutation matrix $B$ for any other node by traversing the tree.

A remarkable property of the mutations of a cluster algebra determined
from such a matrix $B$ is the {\em Laurent phenomenon}. This is a
theorem which states that the cluster variables at any node are
Laurent polynomials in the cluster variables $\bx$ at any other node
\cite{FZlaurent}. This is a highly non-trivial result, because in
general, the mutations \eqref{mutation} give the cluster variables as
a rational function in the cluster variables at another node.

For any given seed $(\bx,B)$ one can ask about the structure of the
evolution on the tree determined by such a seed. If we identify nodes
corresponding to identical seeds, we obtain a quotient graph with a
structure determined by the seed. For example, it is easy to see that
$\mu_i \circ \mu_j = \mu_j\circ \mu_i$ if $B_{ij}=0$. Thus, one can
consider the quotient graph where the nodes corresponding to such an
equivalence relation are identified (see figure \ref{graph} for an
example).

Below we formulate the $Q$-system as a quotient of a subgraph of a
cluster algebra determined by a particularly simple matrix $B$.

\subsection{The $Q$-system as a cluster algebra}\label{Qcluster}
Let $\g$ be a simply-laced, simple Lie algebra of rank $r$, and let
$C$ be its Cartan matrix.

Consider the family of commutative variables $\{Q_{\al,i}\ : \al\in
I_r, i\in\Z\}$ defined by the recursion relation:
\begin{equation}\label{Qsys}
Q_{\al,k+1}= Q_{\al,k-1}^{-1}(Q_{\al,k}^2 - \prod_{\beta\sim\al}
Q_{\beta,k}^{-C_{\al,\beta}}).
\end{equation}
With the initial conditions $Q_{\al,0}=1$ and $Q_{\al,1}=t_\al$ (a
formal variable), this is the $Q$-system \eqref{Qsystem} if one
considers only $k\geq 1$.

In order to make contact with the usual definition of cluster
algebras, it is useful to normalize the system \eqref{Qsys} to get rid
of the minus signs. This is possible for the $Q$-system associated
with any simple Lie algebra.

\begin{lemma}\label{renormalize}
For each simple Lie algebra $\g$, there exists a set of complex
numbers $\{\epsilon_\al\}_{\al\in I_r}$ such that the normalized cluster
variables $\widetilde{Q}_{\al,k}=\epsilon_\al Q_{\al,k}$ satisfy the
normalized 
system:
$$
\widetilde{Q}_{\al,k+1} =\frac{\widetilde{Q}_{\al,k}^2 +
  \prod_{\beta\sim\al} \widetilde{Q}_{\beta,k}}{\widetilde{Q}_{\al,k-1}}.
$$
\end{lemma}
\begin{proof}
From the $Q$-system \eqref{Qsys}, we see that the normalization is 
a choice of $\{\epsilon_\al\}_{\al\in I_r}$ such that
\begin{equation}\label{renorm}
\prod_{\beta}\epsilon_\beta^{C_{\alpha,\beta}}=-1, \quad \forall \alpha.
\end{equation}
Let $\epsilon_\beta = e^{\mu_\beta}$ where $\mu_\beta \in \C$. Since
$C$ is invertible, \eqref{renorm} shows that it is sufficient to set
$\mu_\beta = i\pi \sum_\al C^{-1}_{\beta,\al}$.
\end{proof}
\begin{remark} This is true for the $Q$-system associated with any
  simple Lie algebra \cite{HKOTY}, not just the simply-laced case,
  following an almost identical proof to the one above. One simply 
  notes that the total degree of all characters $Q_{\beta,j'}$
  appearing on the right hand side of the $Q$-system for $Q_{\al,j+1}$
  is $|C_{\al\beta}|$, with $\beta\sim\al$ in each case.\end{remark}

This choice of normalization is not particularly natural from the
point of view of $Q$-systems in general, but it makes the formulation
as a cluster algebra simpler. 

We claim that the $Q$-system corresponds to the evolution of a certain
seed with a certain mutation matrix, defined on a subgraph of the
$n$-ary tree $\mathbb T_n$, where $n=2r$ and $r$ is the rank of $\g$.

\begin{defn} Define the node labeled by $k\in \Z$ as the node
  corresponding to the cluster variable $\bx=(x_1,...,x_{2r})$ given by
\begin{eqnarray}
x_{\al} &=& \widetilde{Q}_{\al,2k},\quad x_{\al+r} = \widetilde{Q}_{\al,2k+1},
\quad (\al\leq r).
\end{eqnarray}
The mutation matrix $B$ at the node $k$ is defined as follows: For all
$\al,\beta\in I_r$,
$$B_{\al+r,\beta} = C_{\al\beta},
\quad B_{\al\beta}=B_{\al+r,\beta+r}=0,\quad B_{\beta\al}=-B_{\al\beta}.$$
\end{defn}

That is, if $C$ is the Cartan matrix of a simply-laced Lie algebra,
the matrix $B$ is
\begin{equation}\label{Bmatrix}
B = \left(\begin{array}{rr} 0 & -C \\ C & 0 \end{array}\right),
\end{equation}
which is skew-symmetric because $C$ is symmetric in this case.

\begin{lemma}
The mutations $\mu_\al$ ($\al\in I_r$), when applied to the seed
$\bx$, commute. That is
$$\mu_\al\circ \mu_\beta(\bx,B) = \mu_\beta \circ \mu_\al(\bx,B),\quad
\al,\beta\in I_r.$$ 
The same statement holds for the mutations
$\mu_{\al+r}$ with $\al\in I_r$:
$$\mu_{\al+r}\circ \mu_{\beta+r}(\bx,B) = \mu_{\beta+r} \circ
\mu_{\al+r}(\bx,B),\quad \al,\beta\in I_r.$$
\end{lemma}
\begin{proof}
This is the result of the fact that the $r\times r$ diagonal blocks of
$B$ vanish. In general, the mutations $\mu_\al$ and $\mu_\beta$
commute whenever $B_{\al,\beta}=0$.
\end{proof}

In fact, some of the mutations from the two sets $\{1,...,r\}$ and
$\{r+1,...,2r\}$ commute, depending on the structure of the Cartan
matrix. However, we do not need this structure in order to describe
the cluster algebra corresponding to the $Q$-system.

We now describe a particularly simple subgraph of the tree $\mathbb
T_{2r}$, whose nodes correspond to cluster variables consisting only
of members of the family $\{Q_{\al,i}\}$, and whose edges correspond
only to mutations describing one of the recursion relations in the
$Q$-system \eqref{Qsys}.

In fact we will consider a quotient of this graph, obtained by
identifing nodes of the subgraph which have the same seed. In
particular, we take a quotient of the graph by the relations
$[\mu_\al, \mu_\beta]=0$ if $\al,\beta\in I_r$ or
$\al-r,\beta-r\in I_r$.

Starting from the node $k$, consider any subset of $I_r$, and traverse
the graph along the edges labeled by these elements. Since the
mutations corresponding to these elements commute, they can be taken
in any order.  Since $\mu_\al^2=1$, they can be assumed to be distinct.

\begin{defn}
The node $k'$ ($k\in \N$) is the node reached from the node $k$ by
traversing edges labeled by all elements of $I_r$. 
\end{defn}
This node is unique and independent of the order of the mutations in
this subset.
\begin{lemma}\label{firstmutation}
The seed corresponding to the node $k'$ is $(\bx',B') = \mu_r\circ \cdots
\circ \mu_1 (\bx,B)$ where
$$
\bx' = \left(\begin{array}{l}\widetilde{Q}_{1,2k+2}\\ \vdots \\
  \widetilde{Q}_{r,2k+2} \\ \widetilde{Q}_{1,2k+1}\\ \vdots \\ \widetilde{Q}_{r,2k+1}
\end{array}\right),\qquad B' = \left( \begin{array}{rr} 0 & C \\ -C &
  0\end{array}\right). 
$$
\end{lemma}
\begin{proof}
The statement about the cluster variable $\bx$ follows from the
definition of the $Q$-system. Note that $\mu_\al(x_\beta)=x_\beta$ if
$\beta\neq \al$ 
and therefore $x'_\beta$ corresponding to the node $k'$ has
$x'_\al = \mu_\al (x_\al) =
\widetilde{Q}_{\al,2k}^{-1}(\widetilde{Q}_{\al,2k+1}^2 +
\prod_{\beta\sim\al} \widetilde{Q}_{\beta,2k+1}) =
\widetilde{Q}_{\al,2k+2}$
where $\al\in I_r$.

Next, we note that if $k,l\neq i$ with $i,l\in I_r$,
then
$\mu_i(B_{k,l})=B_{k,l}$. This is because $B_{il}=0$, whereas
$\mu_i(B_{kl})=B_{kl}+\sgn(B_{ki})[B_{ki}B_{il}]_+ = B_{kl}+0$.

Moreover, $\mu_i(B_{ki})=-B_{ki}$ for $1\leq k\leq 2r$. Since the mutation
$\mu_i$ maintains skew-symmetry so $\mu_i(B_{lk})=B_{lk}$ if $l\in
I_r$ and $k,l\neq i$, and $\mu_i(B_{ik}) = -B_{ik}$ if $k\leq r$.

This shows that
$$
\mu_r\circ \cdots \circ \mu_1 (B ) =\left(\begin{array}{rr} 0 & C \\
  -C & A\end{array} \right)
$$ where $A$ is some matrix to be determined. The following lemma
shows that the matrix $A$ is $0$ if we start from the matrix $B$ as in
\eqref{Bmatrix}
\end{proof}
\begin{lemma}
Define
$$
B = \left( \begin{array}{rr} 0 & -C \\ C & A \end{array}\right),
$$
where $C$ is any Cartan matrix and $A$ is any matrix. Then
$B' := \mu_1 \circ \cdots \circ \mu_r (B)$ has the form
$$
B' = \left( \begin{array}{rr} 0 & C \\ -C & A \end{array}\right).
$$
\end{lemma}
\begin{proof}
According to Lemma \ref{firstmutation}, the transformation $\mu_1\circ
\cdots \circ \mu_r$ changes the sign of the first $r$ rows and columns
of $B$, because $B_{kl}=0$ if $k,l\leq r$. 

We must therefore check that the transformation leaves the matrix $A$
invariant. That is, we must show that for $j,k\in[1,r]$,
\begin{eqnarray*}
0&=&B_{r+j,r+k}'-B_{r+j,r+k} \\ 
&=& \sum_{i=1}^r \sgn(B_{r+j,i})[B_{r+j,i}B_{i,r+k}]_+\\
&=& \sum_{i=1}^r \sgn(C_{j,i})[-C_{j,i}C_{i,k}]_+\\
&=& \sgn(C_{jj})[-C_{jj}C_{jk}]_+ + \sgn(C_{jk}) [-C_{jk}C_{kk}]_+ +
\sum_{i\neq j,k} \sgn(C_{ji})[-C_{ji}C_{ik}]_+ \\
&=& 2 (1-\delta_{jk})|C_{jk}| - 2(1-\delta_{jk})|C_{jk}| + 0 = 0.
\end{eqnarray*}
In the first two terms, we used the fact that the diagonal elements of
$C$ are $2$ and the off-diagonal elements are negative. In the last
term we used the fact that both $C_{ji}$ and $C_{ik}$ are off-diagonal
and hence their product is positive.
\end{proof}

Graphically, the result of applying the mutations $\mu_\al$ ($\al\in
I_r$) sequentially starting with the node $k$ and identifying nodes
corresponding to different orderings of distinct mutations is a
hypercube of dimension $r$,
with two special vertices, $k$ and its mirror image $k'$.
After application of the $r$ distinct mutations $\{\mu_1,...,\mu_r\}$ 
to the node $k$ in any order, one ends
up at the node which we labeled $k'$, which is at the opposite end
of the hypercube. See, For example, figure \ref{graph} for the case of
$\sl_3$.
\begin{figure}[ht]
\begin{center}
\includegraphics[width=10cm]{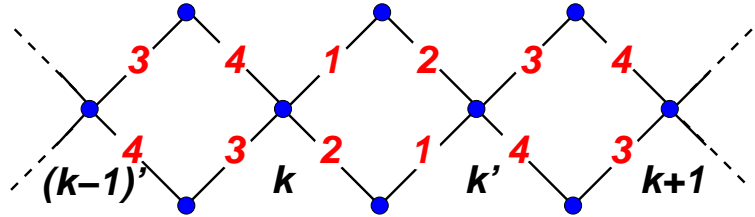} 
\end{center}
\caption{The subgraph $\mathbb T_\g$ for $\sl_3$, where the cluster
  variable $\bx$ is a $4$-vector. The label $i$ on an
  edge denotes the mutation $\mu_i$ relating the two nodes. The
  special nodes $k$, $k'$ and so forth are the nodes at the opposite ends of
  the squares.}
\label{graph}
\end{figure}

Similarly, starting from the node $k\in \N$, we can consider the
subtree of $\mathbb T_{2r}$ 
generated by edges labeled the set $\{\al+r\}$ with
$\al\in I_r$. The corresponding mutations also commute among
themselves. We again take the quotient by identifing nodes with the
same seed. The result is another hypercube, which is also given in
figure \ref{graph}.

An almost identical analysis for this sequence of mutations holds as
for the mutations $\mu_\al$ with $\al\in I_r$. 
Applying the sequence of mutations $\mu_{2r}\circ \cdots \circ
\mu_{r+1}$ to the seed $(\bx,B)_k$ we arrive at the node which we
label $(k-1)'$. We have

\begin{lemma}
The seed corresponding to the node $(k-1)'$ is 
$$
\bx' = \left( \begin{array}{l} \tQ_{1,2k}\\ \vdots \\ \tQ_{r,2k} \\
\tQ_{1,2k-1} \\ \vdots \\ \tQ_{r,2k-1} 
\end{array}\right),\quad B' = \left(\begin{array}{rr} 0 & C \\ -C & 0
\end{array}\right). 
$$
\end{lemma}
The proof is virtually identical to the one above, except that we are
moving ``backwards'' in the $Q$-system, defining $Q_{\al,2k-1}$ from
$Q_{\al,2k+1}$ and $\{Q_{\beta,2k}\}_\beta$ using the $Q$-system.

From the node $k'$, we may now proceed acting with the mutations
$\{\mu_{r+1},...,\mu_{2r}\}$, taken to be distinct and in any
order. The result of applying all $r$ mutations of this type is the
unique node $k+1$, and a similar analysis shows that the seed corresponding
to node $k+1$ has the form $(\bx,B)_{k+1}$ where $B$ is the same as in
equation \eqref{Bmatrix}, and $\bx$ has the same form as $\bx$ at $k$,
but with $k+1$ substituted for $k$. 

In general, we refer to acting ``to the left'' and ``to the right'',
depending on whether we are starting from node $k$ or $k'$. Acting to
the right from a node $k$ means acting with mutations $\mu_i$ with
$i\leq r$, and acting to the left means acting with $\mu_i$ with
$i>r$. Acting to the right on a node $k'$ means acting with mutations
$\mu_i$ with $i>r$, and acting to the left on node $k'$ means acting
with mutations $\mu_i$ with $i\leq r$. 

We can continue to construct this graph in both directions,
encountering only the two types of nodes corresponding to each integer
$k$ after applying the relevant set of $r$ mutations.

Thus we can define the graphs corresponding to the
evolution of the $Q$-systems \eqref{Qsys} and \eqref{Qsystem}. 

\begin{defn}
The graph $\mathbb T_\g$ is obtained from $\mathbb T_{2r}$ by starting
with any node $k\in \Z$, with the matrix $B$ as in \eqref{Bmatrix}, and
considering edges corresponding to the distinct mutations, $1,...,r$, taken
in any order, identifying nodes with the same seeds. 

The node reached after acting once with each mutation in this set 
is the node $k'$. The graph is extended to the right by then acting
with the mutations $\mu_{r+1},...,\mu_{2r}$ to reach node $k+1$, and
so forth.

The graph is extended from node $k$ to the left by acting with the
mutations $r+1,...,2r$, to reach node $(k-1)'$. The graph is extended
to the left from this node by
considering the evolutions $\mu_{1},...,\mu_{r}$ to reach node
$(k-1)$, etc.
\end{defn}
This graph as a chain of hypercubes, infinite in both directions, as in
figure \ref{graph}. By contrast, the
graph corresponding to the original $Q$-system \eqref{Qsystem} is
semi-infinite, with the cutoff achieved by setting initial conditions
at $k=0$.

\begin{defn}
The graph $\overline{\mathbb T}_\g\subset \mathbb T_\g$ corresponding
to the $Q$-system \eqref{Qsystem} is the subgraph of $\mathbb T_\g$
with a specialized seed at $k=0$, where we set $Q_{\al,0}=1$. The
subgraph consists of the node $k=0$ and the nodes to its
``right''. That is, first by the set of mutations $\mu_1,...,\mu_r$,
then the mutations $\mu_{r+1},...,\mu_{2r}$, and alternating these.
\end{defn}
This can be considered as a semi-infinite chain of mutations. See
figure \ref{graph2} for the example of $\sl_3$.
\begin{figure}[ht]
\begin{center}
\includegraphics[width=10cm]{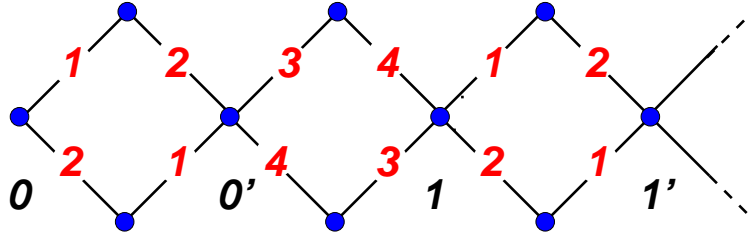} 
\end{center}
\caption{The subgraph $\overline{\mathbb T}_{\sl_3}$.}
\label{graph2}
\end{figure}

The initial condition at $k=0$ implies, in particular, that
$Q_{\al,-1}=0$. 

In fact, all cluster variables $\{\widetilde{Q}_{\al,i}\}_{i\in \Z}$
are found by considering just the nodes $k$ with $k\in \Z$, and
ignoring the fine structure which is given by the Cartan matrix $C$
(in fact, by its rank).
The situation is similar to the ``bipartite belt'' considered by Fomin
and Zelevinsky in the context of $Y$-systems (which are of course
closely connected to $Q$-systems). One may define the compound
mutations
\begin{eqnarray*}
\mu_+&=&\mu_r\circ \cdots \circ \mu_1 \\
\mu_-&=&\mu_{2r}\circ \cdots \circ \mu_{r+1} \\
\end{eqnarray*}
and consider the chain, obtained by acting on the seed at node $k$
with $\mu_+$ and $\mu_-$ alternatively. We have
$(\mu_+)^2=(\mu_-)^2=1$. Acting on the seed at node $k$ with
$\mu_+$ brings us to node $k'$, and acting with $\mu_-$ brings us to
node $(k-1)'$. Similarly acting on node $k'$ with $\mu_-$ brings us to
node $k+1$. The result is a one-dimensional chain with two types of
nodes, see figure \ref{chain}. 
\begin{figure}[ht]
\begin{center}
\includegraphics[width=10cm]{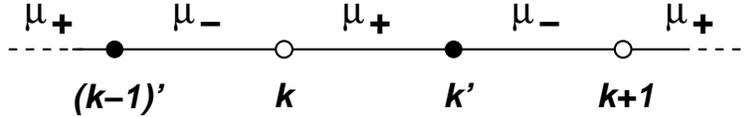} 
\end{center}
\caption{The one-dimensional chain containing all cluster variables.}
\label{chain}
\end{figure}
All cluster variables are contained in the unprimed nodes $k\in
\Z$, or in the primed nodes $k'$.

The cluster variables which give the $Q$-system are contained in the
subgraph with $k\geq 0$, with the cluster variables at $k=0$ given by
$Q_{\al,0}=1$ and $Q_{\al,1}=t_\al$. This is a singular point, because
the evolution $\mu_{\al+r}$ with $\al\in I_r$ is not invertible at
this point, since $\mu_{\al+r}$ has a non-trivial kernel when acting on
$(\bx)_0$. 

\begin{remark}\label{extension}
It is actually possible to make sense of the $Q$-system in the limit
$Q_{\al,0}\to 1$, thus obtaining an infinite chain (rather than a
semi-infinite one), with cluster variables at points $k<0$ which are
$0$ at a finite number of nodes, and are related by a sign to cluster
variables at nodes with $k>0$. We will give details of this extension in a
future publication.
\end{remark}

\subsection{Strong Laurent phenomenon}
One of the most remarkable properties of a cluster algebras, given the
rational nature of the mutations, is the {\em Laurent phenomenon}
\cite{FZlaurent}. This is a theorem which states that any cluster
variable is a Laurent polynomial in the other cluster variables, with
coefficients in the group ring $\Z\mathbb P$, where $\mathbb P$ is the
coefficient ring (the integers in our case).

This is to be compared with the polynomiality property, for the
$Q$-system \eqref{Qsystem} which is explained in Lemma
\ref{polynomiality}. Thus, a corollary of \ref{polynomiality} is the
following theorem.
\begin{thm}\label{poly}
For the cluster algebra corresponding to the $Q$-system above, in
the limit $Q_{\al,0}\to 1$ (that is, $\widetilde{Q}_{\al,0}$ the
appropriate root of unity, see Lemma \ref{renormalize}, the cluster
variables at any node $k>0$ are polynomials in the cluster
variables of the node $k=0$.
\end{thm}

In fact, one can show that the same statement holds for the cluster
variables at $k<0$, corresponding to the extended $Q$-system, see
remark \ref{extension}. Most intriguingly, it appears that the
phenomenon generalizes to other ``branches'' of the tree $\mathbb
T_n$, although this is certainly less than obvious from the
$Q$-system!

This theorem is, in some sense, a stronger statement than the Laurent
phonemonon, because we have no negative powers of $Q_{\al,1}$ in
$Q_{\beta,k}$. However, it says nothing about the dependence on other
cluster variables at other nodes.

We note that it is precisely the Laurent phenomenon theorem that
allows us to prove that the extension of the $Q$-system
\eqref{Qsystem} to negative values of $k$ is well-defined in the limit
$Q_{\al,0}\to 1$, because there can be no pole in the cluster
variables at this point.

\section{Conclusions}
In this note, we have shown that it is possible to formulate the
$Q$-system for simply-laced Lie algebras as a cluster algebra, and
discussed the related graph structure. This gives a strong constraint
on the cluster variables in the limit where $Q_{\al,0}\to 1$,
according to Theorem \ref{poly}. This leaves many open problems.

The graph structure is more interesting, and includes more information
about the Cartan matrix $C$ (other than just its rank) if we allow all
possible evolutions of the $Q$-system from any node in $\mathbb
T_{2r}$. The simplest case where this comes into play is when the rank
is 3 or greater. We have not discussed this structure in these notes.

The formulation of $Q$-systems for the non-simply laced cases is more
complicated and the definitions must be generalized to allow for such
systems. 

In general, it is more natural to consider systems with coefficients,
allowing us to consider the un-normalized $Q$-systems, which contain
minus signs. 

Moreover, we have defined in \cite{DFK} a deformed $Q$-system which
contains more free parameters. We do not yet know whether these can be
considered as cluster algebras, but they are interesting as they have
an alternative formulation for their evolution involving substitution
rather than a rational mutation map.

These and other extensions will be addressed in a future publication.

\vskip.15in

\noindent{\bf Acknowledgements:} The author thanks Ph. Di Francesco,
D. Hernandez, B. Keller, N. Reshetikhin and A. Zelevinsky for their valuable
input, and SPhT at CEA-Saclay
for their hospitality. This research was supported by NSF grant
DMS-05-00759.



\begin{thebibliography}{10}

\bibitem{ChariKR}
Vyjayanthi Chari.
\newblock Minimal affinizations of representations of quantum groups: the rank
  {$2$} case.
\newblock {\em Publ. Res. Inst. Math. Sci.}, 31(5):873--911, 1995.

\bibitem{Chari}
Vyjayanthi Chari.
\newblock On the fermionic formula and the {K}irillov-{R}eshetikhin conjecture.
\newblock {\em Internat. Math. Res. Notices}, (12):629--654, 2001.

\bibitem{DFK}
Philippe Di~Francesco and Rinat Kedem.
\newblock Proof of the combinatorial kirillov-reshetikhin conjecture.
preprint arXiv:0710.4415v1 [math.QA] (to appear in {\em
  Int. Math. Res. Not.} 2008).
\bibitem{FZcluster}
Sergey Fomin and Andrei Zelevinsky.
\newblock Cluster algebras. {I}. {F}oundations.
\newblock {\em J. Amer. Math. Soc.}, 15(2):497--529 (electronic), 2002.

\bibitem{FZlaurent}
Sergey Fomin and Andrei Zelevinsky.
\newblock The {L}aurent phenomenon.
\newblock {\em Adv. in Appl. Math.}, 28(2):119--144, 2002.

\bibitem{FZysys}
Sergey Fomin and Andrei Zelevinsky.
\newblock {$Y$}-systems and generalized associahedra.
\newblock {\em Ann. of Math. (2)}, 158(3):977--1018, 2003.

\bibitem{FZfour}
Sergey Fomin and Andrei Zelevinsky.
\newblock Cluster algebras. {IV}. {C}oefficients.
\newblock {\em Compos. Math.}, 143(1):112--164, 2007.

\bibitem{HKOTY}
G.~Hatayama, A.~Kuniba, M.~Okado, T.~Takagi, and Y.~Yamada.
\newblock Remarks on fermionic formula.
\newblock In {\em Recent developments in quantum affine algebras and related
  topics (Raleigh, NC, 1998)}, volume 248 of {\em Contemp. Math.}, pages
  243--291. Amer. Math. Soc., Providence, RI, 1999.

\bibitem{HKOTT}
Goro Hatayama, Atsuo Kuniba, Masato Okado, Taichiro Takagi, and Zengo Tsuboi.
\newblock Paths, crystals and fermionic formulae.
\newblock In {\em MathPhys odyssey, 2001}, volume~23 of {\em Prog. Math.
  Phys.}, pages 205--272. Birkh\"auser Boston, Boston, MA, 2002.

\bibitem{Her07}
David Hernandez.
\newblock Kirillov-reshetikhin conjecture: The general case. Preprint:
arXiv:0704.2838v3 [math.QA]

\bibitem{Hernandez}
David Hernandez.
\newblock The {K}irillov-{R}eshetikhin conjecture and solutions of
  {$T$}-systems.
\newblock {\em J. Reine Angew. Math.}, 596:63--87, 2006.

\bibitem{KRqsys}
A.~N. Kirillov and N.~Yu. Reshetikhin.
\newblock Representations of {Y}angians and multiplicities of the inclusion of
  the irreducible components of the tensor product of representations of simple
  {L}ie algebras.
\newblock {\em Zap. Nauchn. Sem. Leningrad. Otdel. Mat. Inst. Steklov. (LOMI)},
  160(Anal. Teor. Chisel i Teor. Funktsii. 8):211--221, 301, 1987.

\bibitem{KR}
A.~N. Kirillov and N.~Yu. Reshetikhin.
\newblock Formulas for the multiplicities of the occurrence of irreducible
  components in the tensor product of representations of simple {L}ie algebras.
\newblock {\em Zap. Nauchn. Sem. S.-Peterburg. Otdel. Mat. Inst. Steklov.
  (POMI)}, 205(Differentsialnaya Geom. Gruppy Li i Mekh. 13):30--37, 179, 1993.

\bibitem{KSS}
Anatol~N. Kirillov, Anne Schilling, and Mark Shimozono.
\newblock A bijection between {L}ittlewood-{R}ichardson tableaux and rigged
  configurations.
\newblock {\em Selecta Math. (N.S.)}, 8(1):67--135, 2002.

\bibitem{KS}
Anatol~N. Kirillov and Mark Shimozono.
\newblock A generalization of the {K}ostka-{F}oulkes polynomials.
\newblock {\em J. Algebraic Combin.}, 15(1):27--69, 2002.

\bibitem{Kleber}
Michael Kleber.
\newblock Combinatorial structure of finite-dimensional representations of
  {Y}angians: the simply-laced case.
\newblock {\em Internat. Math. Res. Notices}, (4):187--201, 1997.

\bibitem{KleberPoly}
Michael Kleber.
\newblock Polynomial relations among characters coming from quantum affine
  algebras.
\newblock {\em Math. Res. Lett.}, 5(6):731--742, 1998.

\bibitem{KNS}
Atsuo Kuniba, Tomoki Nakanishi, and Junji Suzuki.
\newblock Functional relations in solvable lattice models. {I}. {F}unctional
  relations and representation theory.
\newblock {\em Internat. J. Modern Phys. A}, 9(30):5215--5266, 1994.

\bibitem{Nakajima}
Hiraku Nakajima.
\newblock {$t$}-analogs of {$q$}-characters of {K}irillov-{R}eshetikhin modules
  of quantum affine algebras.
\newblock {\em Represent. Theory}, 7:259--274 (electronic), 2003.

\bibitem{R}
N.~Yu. Reshetikhin.
\newblock The spectrum of the transfer matrices connected with {K}ac-{M}oody
  algebras.
\newblock {\em Lett. Math. Phys.}, 14(3):235--246, 1987.

\end{thebibliography}
\bibliographystyle{plain}

\def\cprime{$'$} \def\cprime{$'$} \def\cprime{$'$} \def\cprime{$'$}
  \def\cprime{$'$} \def\cprime{$'$} \def\cprime{$'$} \def\cprime{$'$}
  \def\cprime{$'$} \def\cprime{$'$}

\end{document}